\documentclass[11pt]{article}

\usepackage{pgf,tikz}
\usepackage{mathrsfs}
\usetikzlibrary{arrows}

\usepackage{todonotes}
\usepackage{amsmath} 
\usepackage{verbatim}
\usepackage{amssymb, float,empheq}
\usepackage{color}
\usepackage{graphicx,epsfig}
\graphicspath{{figures/}}
\usepackage[applemac]{inputenc}
\usepackage{graphicx}

\usepackage{graphicx} %Loading the package
\graphicspath{{figures/}} %Setting the graphicspat

\usepackage{fullpage, url,amsmath,amsfonts,amssymb,mathtools,mathrsfs,graphicx,  algorithm, float, sansmath,epstopdf,color,caption,enumitem,tabularx}
\usepackage[final]{pdfpages}
\usepackage{amsthm}
\usetikzlibrary{automata,topaths}
\usetikzlibrary{decorations.pathreplacing,shapes.misc}
\usepackage{fancyhdr}

\def\beq{\begin{equation}}
\def\eeq{\end{equation}}
\def\baq{\begin{eqnarray}}
\def\eaq{\end{eqnarray}}
\def\baqn{\begin{eqnarray*}}
\def\eaqn{\end{eqnarray*}}

\usepackage{multirow}
\usetikzlibrary{calc,arrows}
\theoremstyle{plain}
\newtheorem{definition}{Definition}
\newtheorem{remark}{Remark}

\newtheorem{example}{Example}
\newtheorem{theorem}{Theorem}
\newtheorem{lemma}[theorem]{Lemma}

\newtheorem{corollary}[theorem]{Corollary}
\newtheorem{proposition}[theorem]{Proposition}

\usepackage{blindtext}

\usepackage[colorlinks,linkcolor=blue,citecolor=red]{hyperref}

\usepackage{xcolor, framed}

\newcommand{\R}{{\mathbb R}}

\newcommand{\interior}{{\rm int}\kern 0.06em}

\def\<{\langle}
\def\>{\rangle}

\usepackage{ulem}

\usepackage{subcaption}

\usepackage{pict2e}

\if
{

\usepackage[pageref]{backref}
\renewcommand*{\backrefalt}[4]{%
\ifcase #1 %
(Not cited)%
\or
(Cited on p.~#2)%
\else
(Cited on pp.~#2)%
\fi
}

}
\fi

\begin{document}
%\title{Sliding mode observers for set-valued Lur'e  systems subject to uncertainties not in the range of observation}
\title{ {On the Equilibria Computation of   Set-{{Valued}} Lur'e Dynamical Systems}}
\author{Phan Quoc Khanh \thanks{Analytical and Algebraic Methods in Optimization Research Group, Faculty of Mathematics and Statistics, Ton Duc Thang University, Ho Chi Minh City, Vietnam\vskip 0mm
Email: \texttt{phanquockhanh@tdtu.edu.vn}} \qquad 
 Le  Ba Khiet \thanks{Analytical and Algebraic Methods in Optimization Research Group, Faculty of Mathematics and Statistics, Ton Duc Thang University, Ho Chi Minh City, Vietnam\vskip 0mm
 E-mail: \texttt{lebakhiet@tdtu.edu.vn (Corresponding author)}}
 }
 
%\date{}
\maketitle

\begin{abstract}
In this article,   we propose an efficient  way to compute  equilibria of  a general class of set-valued Lur'e  dynamical systems, which plays an important role in the asymptotical analysis of the systems. Besides the equilibria computation, our study can be also used to solve a class of quasi-variational inequalities. Some examples of finding  Nash {quasi-equilibria} in game theory are given.
\end{abstract}



\section{Introduction}
\label{intro}
 Set-valued Lur'e dynamical systems are a  fundamental model in control theory, engineering and applied mathematics (see, e.g., \cite{Acary,ahl,ahl2,br1,bg2,bh,blme,BT,cs,Huang,Huang1,L1,L2,L3,tbc} and the references therein). Although the well-posedness, stability analysis, control and observers for  set-valued Lur'e  systems have been extensively investigated, there has not been much research on the equilibria computation for such systems. It is known that finding equilibria of dynamical systems is important for understanding the asymptotic behavior of the systems. This fact motivates us to consider the equilibria computation for the following general class of  set-valued Lur'e dynamical systems:
 \begin{subequations}
\label{eq:tot}
\begin{empheq}[left={({\mathcal L})}\empheqlbrace]{align}
  & \dot{x}(t) = -f(x(t))+B\lambda(t),\; {\rm a.e.} \; t \in [0,+\infty), \label{1a}\\
  & y(t)=Cx(t)+D\lambda(t), t\geq 0,\\
  &  \lambda(t)   \in -\mathcal{F}(y(t)), \;t\ge 0,\\
  & x(0) = x_0,
\end{empheq}
\end{subequations}
where $H_1, H_2$ are Hilbert spaces, $f: H_1\to H_1$ is monotone  and Lipschitz continuous, $B:H_2\to H_1, C: H_1\to H_2, D: H_2\to H_2$ are linear and bounded, and $\mathcal{F}: H_2\rightrightarrows H_2$ is maximal monotone.  {{Some reformulations of system ($\mathcal L$) can be obtained  as follows. From (1b) and (1c), write $\lambda(\cdot)$ in terms of $x(\cdot)$ as  $\lambda(\cdot)\in (F^{-1}+D)^{-1}Cx(\cdot)$. We have   $Cx(\cdot)\in (D+\mathcal{F}^{-1})(-\lambda(\cdot))$ and so  ($\mathcal L$) is equivalent to the differential inclusion}}
\beq\label{sysm}
\dot{x} \in -\mathcal{H}(x), \;\;\;x(0) = x_0,
\eeq
where $\mathcal{H}(x)=f(x)+B(\mathcal{F}^{-1}+D)^{-1}Cx$. If $D=0$, then $\mathcal{H}(x)=f(x)+B\mathcal{F}(Cx)$, which is ubiquitously considered in the literature on set-valued Lur'e systems. The non-trivial extension $D\neq 0$ can be found widely in electrical circuits (see, e.g., \cite{Acary,ahl,br1,bh,blme,cs,L3,tbc}). It is a natural generalization   since the output $y(t)$ may depend not only on the state $x(t)$ but also the connecting variable $\lambda(t)$.
%One can  find conditions for $f, B, C, D$ (usually based on a condition of passivity) such that $\mathcal{H}$ is a maximal monotone operator (see, e.g., \cite{ahl,br0,cs}).  Then there exists  a unique solution $x$ of $({\mathcal L})$ and  when the time is large the trajectory usually converges weakly to 
{{By the definition, $x^*$ being an equilibrium}} of (\ref{sysm}) {{means}} that 
\beq\label{main}
0 \in f(x^*)+B(\mathcal{F}^{-1}+D)^{-1}(Cx^*),
\eeq
{{i.e.,}} zero belongs to {{the}} sum of two operators: 
\beq
0\in \mathcal{A}(x^*)+\mathcal{B}(x^*),
 \eeq
 where $\mathcal{A}=f$ and $\mathcal{B}=B(\mathcal{F}^{-1}+D)^{-1}C$. {{Inclusion (4) is a stationary condition and $x^*$ is its solution.}} The equilibria computation of (\ref{main}) when $D=0$ and $B=C^T$ was considered in \cite{al2}. Obviously, the extension case $D\neq 0$ makes the analysis more complicated and interesting. 
 Under the well-known  particular passivity condition  $PB=C^T$ for some  symmetric, positive definite operator $P: H_1 \to H_1$, we can reduce our problem to {{finding}} zeros of {{a sum}} of two maximal monotone operators and then we can use  forward-backward algorithms \cite{Attouch0,Attouch,Bauschke,Chen,Tseng} to solve it. This leads  to the problem of computing the  resolvent of the operator $(\mathcal{F}^{-1}+D)^{-1}$ in terms of the resolvent of $F$ (Lemma \ref{comp1}), which is one of our motivations. We also consider the case where the  condition $PB=C^T$ is relaxed   by {{a}} general passivity condition. Let us note that  the fully implicit discretization of $({\mathcal L})$
\beq\label{discf}
\frac{x_{n+1}-x_n}{h} \in -f(x_{n+1})-B(\mathcal{F}^{-1}+D)^{-1}(Cx_{n+1})
\eeq
or semi-implicit discretization
\beq\label{discs}
\frac{x_{n+1}-x_n}{h} \in -f(x_{n})-B(\mathcal{F}^{-1}+D)^{-1}(Cx_{n+1})
\eeq
can {{also be}} reduced to the form of (\ref{main}). 
 
 On the other hand, this technique can not only be used for computing equilibria of Lur'e system $({\mathcal L})$  but also  for solving a class of  {{quasi-variational inequalities}} (QVIs).  Indeed, a particular class of (\ref{main}) when $B=C=I$ and $\mathcal{F}=N_{\Omega}$, the normal cone to a closed convex set $\Omega$, is 
 \beq\label{ree}
 0 \in f(x^*)+(N_\Omega^{-1}+D)^{-1}(x^*).
 \eeq
 We prove that (Proposition \ref{equi}) this class is equivalent to the following quasi-variational inequality: 
  \beq\label{quasi}
 0 \in   f(x^*)+N_\Omega(x^*+Df(x^*))= f(x^*)+N_{K(x^*)}(x^*)
\eeq 
{{involving}} the state-dependent set $K(x):=\Omega-Df(x)$. The term $Df$  plays the role of a perturbation that affects  the state variable under the normal cone operator.
 Applications of QVIs can be found largely in economics, transportation, mechanics, electrical circuits, etc. (see, e.g., \cite{Anh,Bensoussan,Bliemer,Kravchuk,Pang}). If $D=0$, then (\ref{quasi}) becomes the standard variational {{inequality}} problem 
 which is at the core of many constrained optimization problems  and has been intensively investigated in the literature.  It is known that solving quasi-variational inequalities  is far  more difficult than solving variational inequalities  and  {{is possible only}} under quite restrictive conditions, usually based on {{strong monotonicity assumptions}}.  One of the well-known conditions \cite{al1,Nesterov} requires that $f$ is $\mu$-strongly monotone and $L$-Lipschitz continuous, the state-dependent set $K$ is  $l$-Lipschitz continuous, and $l<\mu/L$, which is very small. Then,   quasi-variational inequality (\ref{quasi}) can be  approximated by a sequence of variational inequalities   \cite{Nesterov} or  reduced to a new variational inequality  \cite{al1} to obtain the linear convergence. 
 We show that the  restrictive condition $l<\mu/L$ can be removed in our   class.  It is sufficient to compute the  resolvent of  $(N_\Omega^{-1}+D)^{-1}$ in terms of the resolvent  of $N_\Omega$, which is  the projector operator onto $\Omega$.  If $f$ is strongly monotone,
(\ref{ree}) can be reduced  to a {{strongly monotone}} variational inequality 
  \beq\label{str}
0\in (f^{-1}+D)^{-1}(y^*)+N_\Omega(y^*),
\eeq
where $(f^{-1}+D)^{-1}$ is strongly monotone, single-valued, Lipschitz continuous, and $y^*=x^*+Df(x^*)$.  Solving (\ref{str}) is remarkably easier than solving (\ref{ree}).

The paper is organized as follows.  In Section \ref{sec2}, we recall some definitions and needed results in the theory of monotone operators and passive systems. In Section \ref{sec3},  we provide an efficient way to compute equilibria of set-valued Lur'e systems.  Applications to QVIs and Nash games  are presented in  Section \ref{sec4} and Section \ref{sec5} respectively. Section \ref{sec6} provides some numerical experiments to validate the developed theoretical results.  The paper  ends in Section \ref{sec7} with some conclusions and perspectives.

\section{Notation and preliminaries} \label{sec2}
Let  $H$  be  a {real} Hilbert space  with {the inner product $\langle \cdot,\cdot \rangle$ and the associated norm $\Vert \cdot \Vert$.}  The distance and the projection from a point $x$ to a closed convex subset $C$ of $H$ is defined as follows
$${ d}(x,C):=\inf_{y\in C} \|x-y\|, \;\;{\rm proj}_C(x):={y \in C \;\;{\rm such \;that \;} { d}(x,C)= \|x-y\|}.$$
The normal cone  to  $K$ at $x\in K$ is 
\beq
N_K(x):=\{x^*\in H: \langle x^*, y-x \rangle\le 0\;\;{\rm \;for\;all}\;y\in K\}.
\eeq
  A  function  $f: H\to H$ is called $L$-Lipschitz continuous ($L>0$) if
\beq
\Vert f(x)-f(y)\Vert \le L \Vert x-y \Vert\;\;{\rm \;for\;all}\;x, y\in H.
\eeq
It is called non-expansive if  $L\le 1$  and contractive if $L<1$.  
%It follows that if $f$ is $\mu$-strongly monotone, then 
%\beq
%\Vert f(x)-f(y) \Vert \ge \mu\Vert x-y \Vert.
%\eeq
\noindent The domain, range, graph, and inverse of a set-valued mapping $\mathcal{F}: {H}\rightrightarrows {H}$ are defined respectively by
$${\rm dom}(\mathcal{F})=\{x\in {H}:\;\mathcal{F}(x)\neq \emptyset\},\;\;{\rm rge}(\mathcal{F})=\displaystyle\bigcup_{x\in{H}}\mathcal{F}(x)\;\;$$
 and
 $$\;\;{\rm gph}(\mathcal{F})=\{(x,y): x\in{H}, y\in \mathcal{F}(x)\},\;\;\;\mathcal{F}^{-1}(y)=\{ x\in H: y\in \mathcal{F}(x) \}.$$

\noindent {The mapping $\mathcal{F}$ is called monotone if 
$$
\langle x^*-y^*,x-y \rangle \ge 0 \;\;{\rm \;for\;all}\;x, y\in H, x^*\in \mathcal{F}(x) \;{\rm and}\;y^*\in \mathcal{F}(y).
$$
We say that $\mathcal{F}$ is strongly monotone  if there exists $\mu>0$ such that
\beq
\langle   x^*-y^*, x-y \rangle \ge \mu\Vert x-y \Vert^2\;\;{\rm \;for\;all}\;x, y\in H, x^*\in \mathcal{F}(x) \;{\rm and}\;y^*\in \mathcal{F}(y).
\eeq
Furthermore, if there is no monotone operator $\mathcal{G}$ such that the graph of $\mathcal{F}$ is strictly {contained} in  the graph of $\mathcal{G}$, then $\mathcal{F}$ is called maximal monotone. If $C$ is a nonempty closed, convex set, then $N_C$ is a maximal monotone operator.
 \noindent {The resolvent  of $\mathcal{F}$ is defined  by
$$
J_\mathcal{F}:=(I+\mathcal{F})^{-1}.
$$
}
We say that a linear bounded mapping $E: H\to H$  is 
 \begin{itemize}
\item  $positive$ $semidefinite$, written  $E\ge 0$, if 
$$\langle Ex,x \rangle \ge 0\;\;{\rm\; for\;all}\;x\in H;$$
\item $positive$ $definite$,  written  $E> 0$, if there exists $c>0$ such that 
$$
\langle Ex,x \rangle \ge c\|x\|^2\;\;{\rm\; for\;all}\;x\in H;
$$
\item $semi$-$coercive$  if there exists $c>0$ such that 
$$
\langle Ex,x \rangle \ge c\|x\|^2\;\;{\rm\; for\;all}\;x\in {\rm rge}(E+E^T),
$$
where $E^T$ denotes the transpose of $E$.
\end{itemize}
Next we recall   Minty's theorem in Hilbert spaces (see, e.g., \cite{Brezis}).
\begin{proposition}\label{minty}
Let $H$ be a Hilbert space and $\mathcal{F}: H \rightrightarrows H$ be a monotone operator. Then, $\mathcal{F}$ is maximal monotone if and only if ${\rm rge}(\mathcal{F}+\lambda I)=H$ for some $\lambda>0$.
\end{proposition}
 Let $E$ be a positive definite linear bounded operator. 
 {{We introduce the resolvent  of $\mathcal{F}$ with respect to $E$ as follows}}
$$
J^E_\mathcal{F}:=(E+\mathcal{F})^{-1}.
$$
\begin{proposition}
If $\mathcal{F}: H \rightrightarrows H$ is maximal monotone and $E: H \to H$ is a positive definite linear bounded operator, then  $J^E_\mathcal{F}$ is single-valued with full domain and Lipschitz continuous.
\end{proposition}
\begin{proof}
Since $E: H \to H$ is positive definite and $\mathcal{F}: H \rightrightarrows H$ is  maximal monotone, we deduce that $E+\mathcal{F}$ is maximal monotone and strongly monotone and thus $J^E_\mathcal{F}$ has full domain (Minty's theorem). 
Let $x_i\in H$ and $y_i\in J^E_\mathcal{F}x_i, i=1, 2$. We have 
$$
x_i\in \mathcal{F}(y_i)+Ey_i.
$$
There exists some $c>0$ such that
$$
\Vert x_1-x_2 \Vert \Vert y_1-y_2  \Vert \ge \langle x_1-x_2, y_1-y_2 \rangle\ge c \Vert y_1-y_2\Vert^2,
$$
which implies that $J^E_\mathcal{F}$ is single-valued and Lipschitz continuous.
\end{proof}
\begin{remark}
In general, if we do not have the explicit form of $J^E_\mathcal{F}$, we can compute it easily. Indeed let $y=J^E_\mathcal{F}x=(E+\mathcal{F})^{-1}(x)$. Then, $0\in \mathcal{F}(y)+Ey-x$. Let $g(y):=Ey-x$. Then, $g$ is single-valued, Lipschitz continuous, and strongly monotone. Then, we can use the forward-backward algorithms (see, e.g., \cite{Tseng}) to compute $y$ for a given $x$ with linear convergence rate.  
\end{remark}

{{For Hilbert spaces $H_1,H_2$, let}} a Lipschitz continuous mapping $f: {H_1\to H_1}$ and linear bounded operators $B:{H_2\to H_1}, C: {H_1\to H_2} $, $D: {H_2\to H_2}$ be given. We generalize the notion of passive systems  (see, e.g., \cite{blme}) for nonlinear $f$ as follows.
\begin{definition} \label{passi} The system $(f,B,C,D)$ is called $P$-$passive$ if there exists a symmetric positive linear bounded mapping $P: {H_1\to H_1}$ such that  for all $x_1,x_2\in H_1$ and  $y_1,y_2\in H_2$, one has
\beq\label{passi}
\langle P(f(x_1)-f(x_2)), x_1-x_2 \rangle+ \langle (PB-C^T) (y_1-y_2), x_1-x_2 \rangle+ \langle D(y_1-y_2),y_1-y_2 \rangle \ge 0.
\eeq
\end{definition}
\begin{remark}
i) If $f, B, C, D$ are  matrices, then $(f,B,C,D)$ is $P$-passive if and only if  the matrix 
$$\left( \begin{array}{cc}
 (Pf+f^TP)& \;\;\;\;PB-C^T\\ \\
B^TP-C & (D+D^T)
\end{array} \right)$$
is positive semidefinite  \cite{blme}. \\
ii) If $(f,B,C,D)$ is $P$-passive, then $Pf$ and $D$ are monotone. 
\end{remark}

\section{Main results}\label{sec3}

  In this section, we provide a method to compute an equilibrium point of the Lur'e set-valued dynamical system $({\mathcal L})$, i.e., a {{zero point}}  of the inclusion  (\ref{main}). First, let us suppose that the following assumptions  are satisfied, which are usually used in the literature on set-valued Lur'e systems (see, e.g., \cite{ahl,bg2,bh,BT,cs,L1,tbc}).\\

\noindent$\mathbf{Assumption \;1}:$ The operator $ f: H_1 \to H_1$ is $L_f$-Lipschitz continuous. The operators $B: H_2 \to H_1, C: H_1 \to H_2, D: H_2 \to H_2$ are linear bounded.  There exists a symmetric, positive definite linear bounded operator $P: H_1 \to H_1$ such that $PB=C^T$. In addition, $Pf$ and $D$ are monotone.\\

%\noindent$\mathbf{Assumption \;2}:$ $\Omega$ is a nonempty closed convex subset of $H$.\\
%
%\noindent$\mathbf{Assumption \;3}:$ The system $(f,B,C,D)$ is $P$-passive, where $f$ is $L_f$-Lipschitz continuous and $D$ is semi-coercive.\\

\noindent$\mathbf{Assumption \;2}:$ The set-valued operator $\mathcal{F}: H_2 \rightrightarrows H_2$ is maximal monotone. \\

\noindent$\mathbf{Assumption \;3}:$   $0\in {\rm ri}({\rm rge}({C})-{\rm rge}({\mathcal{F}}^{-1}+D))$, where ${\rm ri}(S)$ denotes the relative interior of set $S$.
\begin{remark}
 It is easy to see that, under Assumption 1, the system $(f,B,C,D)$ is  $P$-$passive$. 
\end{remark}

%\noindent$\mathbf{Assumption \;5}:$  If  $Cx\in {\rm rge}({F}^{-1}+D)$ then  $ {\rm rge}(D+D^T) \cap (F^{-1}+D)^{-1}Cx\neq \emptyset$.  \\
%
%\noindent$\mathbf{Assumption \;6}:$   $0\in {\rm ri}({\rm rge}({C})-{\rm rge}({F}^{-1}+D))$, where ${\rm ri}(S)$ denotes the relative interior of the set $S$.\\

The following result provides a way to compute the resolvent of $\mathcal{B}:=(\mathcal{F}^{-1}+D)^{-1}$ based on the resolvent  of $\mathcal{F}$ with respect to a positive definite linear operator, which can be used to  solve (\ref{main}) in general.

 \begin{lemma}\label{comp1}
Let $x\in H$ and $\gamma>0$. Then, $y=J_{\gamma \mathcal{B}}(x)$ if and only if  $y=E(\gamma J_\mathcal{F}^E(Ex)+Dx)$, where $E:=(\gamma I+D)^{-1}$ is a positive definite linear bounded operator. 
\end{lemma}
\begin{proof} First we check that $E:=(\gamma I+D)^{-1}$ is a positive definite linear bounded operator. Since $D$ is monotone, we deduce that $\gamma I+D$ is strongly monotone and $E$ is well-defined and single-valued. Let $x_i=Ey_i = (\gamma I+D)^{-1}y_i, i= 1,2,$ and $\lambda \in \R$. One has $y_i= \gamma x_i+Dx_i$ and thus $y_1+\lambda y_2=\gamma (x_1+\lambda x_2)+D(x_1+\lambda x_2)$. Therefore, $E(y_1+\lambda y_2)=x_1+\lambda x_2$, which implies that $E$ is linear. For all $y\in H$, let $x=Ey=(\gamma I+D)^{-1}(y)$, i.e., $y=\gamma x+Dx$.  Then, $\Vert y \Vert^2=\gamma \langle x, y \rangle +\langle Dx, y \rangle \le \gamma \Vert x \Vert \Vert y \Vert+\Vert D \Vert \Vert x \Vert \Vert y \Vert \Rightarrow \Vert y \Vert \le (\gamma+\Vert D \Vert) \Vert x \Vert$. Hence,
$$
\langle Ey,y \rangle = \langle x, \gamma x+Dx \rangle \ge\gamma \Vert x\Vert^2 \ge \frac{\gamma}{(\gamma+\Vert D \Vert)^2} \Vert y\Vert^2.
$$
Thus, $E$ is positive definite. 
 We have  the following chain of equivalent relations 
 \baqn
 &&y= J_{\gamma \mathcal{B}}(x)=(I+\gamma \mathcal{B})^{-1}(x) \Leftrightarrow x\in y +\gamma (\mathcal{F}^{-1}+D)^{-1}(y)  \\
 &\Leftrightarrow& \frac{x-y}{\gamma}\in (\mathcal{F}^{-1}+D)^{-1}(y) \Leftrightarrow y \in \mathcal{F}^{-1}(\frac{x-y}{\gamma})+D \frac{x-y}{\gamma} \\
  &\Leftrightarrow&   \frac{x-y}{\gamma}\in \mathcal{F}(y+D\frac{y-x}{\gamma}) \Leftrightarrow 0\in h(y)+ \mathcal{F}(y+Dh(y)),
 \eaqn
  where  $h(y)=\frac{y-x}{\gamma}$. Note that $h$ is {{$\frac{1}{\gamma}$-}}strongly monotone and $\frac{1}{\gamma}$-Lipschitz continuous. Let $z=y+Dh(y)=\gamma h(y)+x+Dh(y)$, which implies that $h(y)=(\gamma I+D)^{-1} (z-x)=E(z-x)=Ez-Ex$ since $D$ is monotone and linear.  Then, the last inclusion is equivalent to 
  \baqn
  0\in Ez-Ex+\mathcal{F}(z), \,{\rm i.e.,} \,z=J_\mathcal{F}^E(Ex).
  \eaqn
  On the other hand,
  $$
  z=y+Dh(y)=y+D\frac{y-x}{\gamma},$$
  which is equivalent to  $$y=(\gamma I+{D})^{-1}(\gamma z+Dx)=E(\gamma z+Dx).
  $$
  The conclusion follows. 
 \end{proof}
  Note that if the resolvent of $\mathcal{B}=(\mathcal{F}^{-1}+D)^{-1}$ is available, then we can compute the resolvent of $C^T(\mathcal{F}^{-1}+D)^{-1}C$ by using the fixed-point approach (see, e.g., \cite{al2,Micchelli}). The following algorithm is based on the Tseng algorithm \cite{Tseng} applied to our problem (\ref{main}). Let $\tilde{f}:= Pf$ and $G:={C}^T(\mathcal{F}^{-1}+D)^{-1} {C}={C}^T \mathcal{B} C$. Since $D$ is monotone, from Assumptions 2, 3 we deduce that $G$ is a maximal monotone operator (see, e.g., \cite{Pennanen}).\\
  
 \noindent \textbf{Algorithm \;1}: Compute $J_{\gamma \mathcal{B}}$ and then $J_{\gamma G}$. Given $x_0\in H$, for $n\ge 0$, we compute \\
 \beq \left\{
\begin{array}{lll}
 y_{n}=J_{\gamma G}\Big(x_n- \gamma \tilde{f}(x_n)\Big),\\
&&\\
x_{n+1}= y_n- \gamma (\tilde{f}(y_n)-\tilde{f}(x_n)). \\
\end{array}\right. \label{sysh}
\eeq

 \begin{theorem}\label{mthe}
  Let Assumptions 1--3 hold. Let $(x_n)$ be the sequence generated by  Algorithm 1. Then, $(x_n)$ converges weakly to an equilibrium point $x^*$ of the Lur'e set-valued dynamical system $({\mathcal L})$, i. e., $x^*$ satisfies  the inclusion  (\ref{main}).
 \end{theorem}
  \begin{proof}
We have the chain of equivalent relations
  \baqn
&&0 \in f(x^*)+B(\mathcal{F}^{-1}+D)^{-1}(Cx^*) \Leftrightarrow 0 \in Pf(x^*)+PB(\mathcal{F}^{-1}+D)^{-1}(Cx^*)\\
&\Leftrightarrow& 0 \in Pf(x^*)+C^T(\mathcal{F}^{-1}+D)^{-1}(Cx^*)\\
&\Leftrightarrow& 0 \in {{\tilde{f}(x^*)+G(x^*),}}
\eaqn
where $\tilde{f}:= Pf$ and $G:={C}^T(\mathcal{F}^{-1}+D)^{-1} {C}$. Note that $Pf$ is monotone, Lipschitz continuous and $G$ is a maximal monotone operator. The conclusion follows (see, e.g.,  \cite[Theorem 3.6]{Shehu} or \cite[Theorem 3.4]{Tseng}). 
 \end{proof}
   \begin{remark}
 One can use Algorithm 1 to obtain the numerical trajectory of the Lur'e system $(\mathcal{L})$ by using the fully implicit scheme (\ref{discf}) or semi-implicit scheme (\ref{discs}). On the other hand, the originality of Algorithm 1 is based on the semi-implicit scheme (\ref{discs}).
   \end{remark}
   Next, we show that the obtained result can be extended to general passive systems. \\
   
   \noindent$\mathbf{Assumption \;1'}:$  The system $(f,B,C,D)$ is $P$-passive, where $ f: H_1 \to H_1$ is $L_f$-Lipschitz continuous,   $B:{H_2\to H_1}, C: {H_1\to H_2}, D: {H_2\to H_2}$ are linear bounded operators, and $D$ is semi-coercive.  \\
   
   \noindent$\mathbf{Assumption \;4}:$  There exists $y_0\in H_2$ such that   if  $Cx\in {\rm rge}({\mathcal{F}}^{-1}+D)$, then  $ (y_0+{\rm rge}(D+D^T) )\cap (\mathcal{F}^{-1}+D)^{-1}Cx\neq \emptyset$. 
   \begin{remark}
 1) In finite dimensional spaces, the semi-coercivity of $D$ holds automatically since $D$ is positive semidefinite, which comes from the passivity of $(f,B,C,D)$.\\
 2) Assumption \;4 is an improvement of the usual assumption with $y_0=0$ used in the literature (see, e.g., \cite{L1,L2,tbc}).
\end{remark}
 \begin{lemma} \label{kersub}
  Let Assumption 1$^\prime$ hold. Then, $\ker(D+D^T)\subset \ker(PB-C^T)$.
   \end{lemma} 
   \begin{proof}
   Let $y\in \ker(D+D^T)$. We have 
   $$
   \langle Dy,y\rangle=0.
   $$
   Suppose that $y\notin \ker(PB-C^T)$. Then, $y^*:=(PB-C^T)y\neq 0$.   In (\ref{passi}), we take $y_1=y, y_2=0$, and $x_1-x_2=-y^*/n$. Note that 
   $$
   \langle P(f(x_1)-f(x_2)), x_1-x_2 \rangle \le \Vert P \Vert L_f \Vert x_1-x_2 \Vert^2=\frac{L_f \Vert P \Vert \Vert y^* \Vert^2 }{n^2},
   $$
   where $L_f$ is the Lipschitz constant of $f$. 
   
   Thus, we have 
   $$
   \frac{L_f \Vert P \Vert \Vert y^* \Vert^2 }{n^2}-\frac{ \Vert y^* \Vert^2}{n}\ge 0.
   $$
   Letting $n\to +\infty$, we obtain a contradiction. Hence,  $y\in \ker(PB-C^T)$ and the conclusion follows. 
   \end{proof}
    \begin{lemma} \label{kernel}
    Suppose that $D$ is semi-coercive. If $y\in H_2$ satisfies $\langle Dy,y \rangle=0$, then $y\in \ker(D+D^T)$.
       \end{lemma}
  \begin{proof}
  Let $y=y_1+y_2$, where $y_1, y_2$ are the projections of $y$ onto $\ker(D+D^T)$ and ${\rm rge}(D+D^T)$ respectively. We have 
  $$
  \langle (D+D^T) y, y \rangle= 2 \langle D y, y \rangle = 0,
  $$
  which implies that 
  $$
   \langle (D+D^T) y_2, y_2 \rangle=  0. 
  $$
  Since $D$ is semi-coercive, we infer that $y_2=0$ and hence $y=y_1\in \ker(D+D^T)$.
  \end{proof}      
 \begin{lemma} \label{glip}
  Let Assumptions 1', 2, 3 hold. {{Then, $g$ defined by $g(x):=Pf(x)+(PB-C^T)(\mathcal{F}^{-1}+D)^{-1}Cx$ is single-valued}}. In addition, if   Assumption 4 is satisfied, then  $g$ is Lipschitz continuous with the Lipschitz constant 
  $$
  L_g=\Vert P \Vert  L_f+\frac{\Vert C \Vert \Vert PB-C^T \Vert}{c}.
  $$
   \end{lemma}    
  \begin{proof}
First we prove that $(PB-C^T)(\mathcal{F}^{-1}+D)^{-1}Cx$ is {{a singleton for every $x\in H_1$. For}} $y_1, y_2 \in (\mathcal{F}^{-1}+D)^{-1}Cx$, we have $y_i\in \mathcal{F}(Cx-Dy_i)$. The monotonicity of $F$ deduces that 
$$\langle y_1-y_2, D(y_1-y_2) \rangle \le 0.$$
 Since $D$ is monotone, we have $\langle y_1-y_2, D(y_1-y_2) \rangle = 0$ and thus $y_1-y_2\in \ker(D+D^T)\subset \ker(PB-C^T)$ (Lemma \ref{kernel}).
 Therefore, $(PB-C^T)(y_1-y_2)=0$ which implies that $(PB-C^T)(\mathcal{F}^{-1}+D)^{-1}Cx$ is {{a singleton}}.

Now suppose that Assumption 4 holds. Let $x_1, x_2 \in {\rm dom}((\mathcal{F}^{-1}+D)^{-1}C)$ or equivalently $Cx_1, Cx_2 \in {\rm rge}({\mathcal{F}}^{-1}+D)$. Then, we can find $y_1, y_2 \in {\rm rge}(D+D^T)$ such that $y_i+y_0 \in (\mathcal{F}^{-1}+D)^{-1}Cx_i,\; i=1, 2$ {{($y_0$ is given in Assumption 4)}}. Similarly as above, we have 
$$
\langle y_1-y_2, C(x_1-x_2) -D(y_1-y_2) \rangle \ge 0.
$$
Since $D$ is semi-coercive, there exists $c>0$ such that 
\baq \nonumber
c\Vert y_1-y_2\Vert^2 &\le& \langle y_1-y_2, D(y_1-y_2) \rangle \le \langle y_1-y_2, C(x_1-x_2) \rangle \\
&\le& \Vert C \Vert \Vert x_1-x_2 \Vert \Vert y_1-y_2 \Vert, \label{semip}
\eaq
which implies that 
$$
\Vert y_1-y_2\Vert \le \frac{\Vert C \Vert}{c} \Vert x_1-x_2 \Vert .
$$
Thus,
\baqn
\Vert g(x_1)-g(x_2)\Vert &=& \Vert Pf(x_1)-Pf(x_2)+(PB-C^T) (y_1+y_0-y_2-y_0)\Vert \\
&\le & (\Vert P \Vert  L_f+\frac{\Vert C \Vert \Vert PB-C^T \Vert}{c}) \Vert x_1-x_2 \Vert
\eaqn
and the conclusion follows. 
 \end{proof}  
 
 \begin{lemma} 
Let Assumptions 1', 2, 3, 4 hold and
\beq
\mathcal{H}:=Pf+PB(\mathcal{F}^{-1}+D)^{-1}C.
\eeq
Then, $\mathcal{H}$ is a maximal monotone operator. 
   \end{lemma} 
 \begin{proof}
 First we prove that $\mathcal{H}$ is monotone. Let 
 $x_1, x_2 \in {\rm dom}(\mathcal{H})$ and
  $z_i \in \mathcal{H}(x_i), i=1, 2$. Then, we have 
 $z_i=Pf(x_i)+$ {{$PBy_i$}}, where  $y_i\in (\mathcal{F}^{-1}+D)^{-1}Cx_i$, {{i.e.,}} $y_i\in \mathcal{F}(Cx_i-Dy_i)$, and thus 
 $$
\langle y_1-y_2, C(x_1-x_2) -D(y_1-y_2) \rangle \ge 0.
$$
Using the $P$-passivity of $(f,B, C, D)$, we have
\baqn
\langle z_1-z_2, x_1-x_2 \rangle &=& \langle Pf(x_1)-Pf(x_2)+PB (y_1-y_2), x_1-x_2 \rangle\\
& \ge & \langle y_1-y_2, C(x_1-x_2) -D(y_1-y_2) \rangle \\
&\ge & 0.
\eaqn
Thus, $\mathcal{H}$ is monotone. 
 Note that $g(x):=Pf(x)+(PB-C^T)(\mathcal{F}^{-1}+D)^{-1}Cx$ is Lipschitz continuous (Lemma \ref{glip}). We can write $\mathcal{H}=Pf+(PB-C^T)(\mathcal{F}^{-1}+D)^{-1}C+C^T(\mathcal{F}^{-1}+D)^{-1}C=g+G$. We use Minty's theorem to prove that $\mathcal{H}$ is a maximal monotone operator with $\lambda = 2L_g$.  For each $y\in H_1$, we prove that there exists $x\in H_1$ such that 
 $$
 y\in \mathcal{H}(x)+\lambda x=G(x)+(g(x)+L_g x)+L_g x.
 $$
Since $G$ is maximal monotone and $g+L_g I$ is Lipschitz and monotone, we conclude that the operator $G+g+L_g I$ is maximal monotone  \cite[Lemma 2.4]{Brezis}. Then, the existence of $x$ is guaranteed by using Minty's theorem and the conclusion follows.  
 \end{proof}  
 
  Since $g$ may be nonmonotone, the convergence may not be obtained if we use the splitting algorithms. However, we can compute the resolvent  $J_{\gamma \mathcal{H}}$ for some small $\gamma >0$ (see, e.g., \cite{LT}). \\

   \noindent \textbf{Algorithm \;2}: For some small $\gamma >0 $, compute $J_{\gamma \mathcal{B}}$, $J_{\gamma G}$ and then $J_{\gamma \mathcal{H}}$. Given $x_0\in H$, for $n\ge 0$, we compute \\
 $$
 x_{n+1}=J_{\gamma \mathcal{H}}(x_n).
 $$

   \begin{theorem}
  Let Assumptions 1', 2, 3, 4 hold. Let $(x_n)$ be the sequence generated by  Algorithm 2. Then, $(x_n)$ converges weakly to an equilibrium point $x^*$ of the {{set-valued Lur'e  system}} $({\mathcal L})$.
 \end{theorem} 
\begin{proof}
Since $\mathcal{H}$ is a maximal monotone operator, using the proximal point algorithm we obtain the conclusion (see, e.g., \cite{Rockafellar}). 
\end{proof}

\section{Applications to quasi-variational inequalities}\label{sec4}
 Let us consider a particular case of   (\ref{main}) when $B=C=I$:
  \beq\label{reeg}
 0 \in f(x^*)+(\mathcal{F}^{-1}+D)^{-1}(x^*).
 \eeq
 In addition, if  $\mathcal{F}=N_{\Omega}$, where $\Omega$ is a closed convex set, then (\ref{reeg})  becomes 
 \beq\label{reek}
 0 \in f(x^*)+(N_\Omega^{-1}+D)^{-1}(x^*).
 \eeq
 We show that (\ref{reek}) is indeed a quasi-variational inequality (QVI). 
\begin{proposition} \label{equi}The monotone inclusion (\ref{reek}) is equivalent to the QVI:
  \beq\label{quasi1}
 0 \in   f(x^*)+N_{K(x^*)}(x^*),
\eeq 
where $K(x):=\Omega -Df(x)$.
\end{proposition}

\begin{proof}
{{We have the following equivalences}}
$0 \in f(x^*)+(N_\Omega^{-1}+D)^{-1}(x^*) \Leftrightarrow 0 =  f(x^*)+\lambda$ {{with}} $\lambda \in {{(N_\Omega^{-1}}}+D)^{-1}(x^*)\Leftrightarrow \lambda \in N_\Omega(x^*-D\lambda)$. From $\lambda=-  f(x^*)$, we obtain 
$$
 0 \in  f(x^*)+N_{\Omega}(x^*+Df(x^*))= f(x^*)+N_{K(x^*)}(x^*)
$$
and the conclusion follows. 
\end{proof}
\begin{remark}
(i) Similarly, the monotone inclusion (\ref{reeg}) is equivalent to
  \beq\label{reeg1}
 0 \in f(x^*)+\mathcal{F}(x^*+Df(x^*).
 \eeq
(ii) The mapping $f$ is not required to be $\mu$-strongly monotone and  the moving set $K(x)$ in the QVI (\ref{quasi1}) is not necessarily   $l$-contractive with $l<\mu/L$ as in \cite{al1,Nesterov}.  Indeed, to solve the QVI (\ref{quasi1}), we can solve it equivalent form, the monotone inclusion (\ref{reek}), by computing the resolvent of the operator $\mathcal{C}:=(N_\Omega^{-1}+D)^{-1}$. The following is a consequence of Lemma \ref{comp1}.
\end{remark}
 \begin{corollary}\label{comp2}
Let be given $x\in H$ and $\gamma>0$. Then, $y=J_{\gamma \mathcal{C}}(x)$ {{if and only if }}  $y=E(\gamma J_{N_\Omega}^E(Ex)+Dx)$, where $E:=(\gamma I+D)^{-1}$ is a positive definite linear bounded operator. 
\end{corollary}

Next we consider the case where $f$ is strongly monotone. First, we recall {{the following fact (see, e.g.,    \cite{IT}), which will be used below}}.

\begin{lemma}\label{inve}
If $f$ is $\mu$-strongly monotone and $L$-Lipschitz continuous, then $f^{-1}$ is $\frac{\mu}{L^2}$-strongly monotone and $\frac{1}{\mu}$-Lipschitz continuous.
\end{lemma}

\begin{proof} For the first property of $f^{-1}$, see, e.g., \cite{IT}. The Lipschitz continuity of $f^{-1}$ follows from the observation that for all $x,y\in H_1$, one has
\[\Vert x-y \Vert=\Vert f\circ f^{-1}(x)-f\circ f^{-1}y \Vert\ge \mu \Vert f^{-1}(x) -f^{-1}(y)\Vert.
\]		
\end{proof}

The following result shows that if $f$ is strongly monotone, then  the QVI (\ref{quasi1}) can be reduced into a strongly monotone  Variational Inequality.

\begin{proposition} \label{charac}
If $f$ is strongly monotone, then $x^*$ is a solution of the QVI (\ref{quasi1}) if and only if $y^*=(f^{-1}+D)f(x^*)=x^*+Df(x^*)$ is the unique solution of the  strongly monotone Variational Inequality
\beq\label{dual}
0\in (f^{-1}+D)^{-1}(y^*)+N_\Omega (y^*).
\eeq
\end{proposition}
\begin{proof}
Suppose that $f$ is strongly monotone and $0 \in  f(x^*)+{{N_\Omega}}(x^*+Df(x^*))$. Then, $f^{-1}+D$ is strongly monotone, Lipschitz continuous and so is $(f^{-1}+D)^{-1}$.  Let $y^*=x^*+Df(x^*)=(f^{-1}+D)f(x^*)$, which implies that $f(x^*)=(f^{-1}+D)^{-1}(y^*)$. Then, we obtain the conclusion.
  \end{proof}

\begin{remark}
The inclusion (\ref{dual}) can be considered as the dual form of (\ref{reek}). 
\end{remark}

\section{Applications to Nash Games} \label{sec5}
Now, we give some examples in finding Nash {{quasi-equilibria}} using our developed results.
 Let us consider a game with $m$ players. {{Each player $i\in S:=\{1,2,\ldots,m\}$ chooses a strategy vector in a set of available strategies $C^i(x^{-i})$, which depends on the strategies chosen by the other players $x^{-i}:=(x^1,\ldots,x^{i-1},$ $x^{i+1},\ldots,x^m )$.}} 
  The cost function $g_i(x^i,x^{-i})$ of {the} player $i$ depends   on both his strategy $x^i$ and the strategies of {the} other players $x^{-i}$.
\begin{definition}
A  Nash {{quasi-equilibrium}} of the game is a vector $x=(x^1,x^2,\ldots,x^m)$ such that
\beq\label{Nash}
x^i\; \;\;{\rm minimizes}\;\;\; g_i(\cdot,x^{-i})\;\;\;{\rm subject\;to}\;\;x^i\in C^i(x^{-i}).
\eeq
\end{definition}

\begin{example}
Suppose that for each $i\in S$, the function $g_i: \R^{n\times m}\to \R$  has the separable  form $g_i(x^i,x^{-i})=\langle g_{i1}(x^{-i}), x^i \rangle+g_{i2}(x^{-i})$, where $g_{i1}: \R^{n\times (m-1)}\to \R^n$  is  Lipschitz continuous and  $g_{i2}: \R^{n\times (m-1)}\to \R$. In addition,  the strategy $x^i$ is restricted to a convex set $C^i(x^{-i}):=K_i-c_i  g_{i1}(x^{-i})$ where $K_i$ is convex and $c_i\ge0$.  Let $D$ be the diagonal matrix $D={\rm diag}(c_1,c_2,\ldots,c_m)$, 
$
\Omega=\prod_{i=1}^{m}K_i
$
and
$$
f(x)=(  g_{11}(x^{-1}),\ldots,   g_{m1}(x^{-m}))^T.
$$
Note that $(\ref{Nash})$ is equivalent to 
\baqn
0\in  g_{i1}(x^{-i})+N_{K_i-c_i  g_{i1}(x^{-i})}(x^i)= g_{i1}(x^{-i})+N_{K_i}(x^i+c_i  g_{i1}(x^{-i})),\;\;\forall \;i=1, 2, \ldots,m.
\eaqn
Consequently, we have 
\beq
0\in f(x)+N_\Omega(x+Df(x)),
\eeq
which is in the form of $(\ref{quasi1})$ and equivalently, in the form of (\ref{reek}).
\end{example}
\begin{example}
Next, assume that $g_i$ can be decomposed as 
\beq\label{decompo}
g_i(x^i,x^{-i})=f_i(x^i,x^{-i})+h_i(x^i+u_i(x^{-i})),
\eeq
where 
\begin{itemize}
\item ${f_i: \R^{n\times m}}\to {\R}$ {is such that}  $f_i(\cdot,x^{-i})$ is convex and differentiable;\\
\item $u_i: \R^{n\times (m-1)}\to {\R^n}$ is a Lipschitz continuous {mapping} which acts as perturbation  of the strategies $x^{-i}$ of the other players,  affecting the strategy $x^i$;\\
\item $h_i: \R^{n}\to {\R}$ is a convex function.
\end{itemize}
Then, $(\ref{Nash})$ is equivalent to 
\beq\label{sufcon}
0\in \nabla_{x^i} f_i(x^i,x^{-i})+\partial h_i(x^i+u_i(x^{-i})),\;\; i= 1,2\ldots,m.
\eeq
Let 
$$
f(x)=( \nabla_{x^1} f_1(x^1,x^{-1}),\ldots,  \nabla_{x^m} f_m(x^m,x^{-m}))^T,
$$
$$
F(x)=(\partial h_1(x^1),\ldots, \partial h_m(x^m))^T,
$$
and 
$$
u(x)=(u_1(x^{-1}), \ldots, u_m(x^{-m}))^T.
$$
Then, $F: \R^{n\times m}\to \R^{n\times m}$ is a maximal monotone operator and  (\ref{sufcon}) can be rewritten as 
\beq\label{statec}
0\in f(x)+F(x+u(x)).
\eeq
Suppose that for each $i$, the function $f_i$ is linear with respect to $x_i$, i.e., $f_i(x^i,x^{-i})=\langle f_{i1}(x^{-i}), x^i \rangle+f_{i2}(x^{-i})$,  where $f_{i1}: \R^{n\times (m-1)}\to \R^n$  is  Lipschitz continuous and  $f_{i2}: \R^{n\times (m-1)}\to \R$ and $u_i(x^{-i})=d_i f_{i1}(x^{-i})$ for some $d_i\ge0$.  Let  $D={\rm diag}(d_1,d_2,\ldots,d_m)$. Then, (\ref{statec}) becomes
\beq\label{stat}
0\in f(x)+F(x+Df(x)),
\eeq
which can be rewritten in the form of $(\ref{reeg1})$.
\end{example}

\section{Numerical Example}\label{sec6}
In this section, we give a simple example in $\R^2$  to illustrate our theoretical results. 
 Let us consider the system $(\mathcal L)$ with 
  $$
  f=\left( \begin{array}{ccc}
9 &\;\; -1 \\ \\
1&\;\;8
\end{array} \right), B=C=I, 
 D = \left( \begin{array}{ccc}
0 &\;\; 0 \\ \\
0 &\;\; 1 
\end{array} \right), F(x)=\left( \begin{array}{ccc}
{\rm Sign}(x_1) \\ \\
{\rm Sign}(x_2)
\end{array} \right),
  $$
where  
\[
{\rm Sign}(\alpha) = \begin{cases}
1 & \text{if } \alpha > 0, \\
[-1, 1] & \text{if } \alpha = 0, \\
-1 & \text{if } \alpha < 0.
\end{cases}
\]
In practice, to simplify the computation, one usually uses the explicit scheme and replaces the ${\rm Sign}$ function by the classical sign function as follows
 \begin{subequations}
\label{eq:tot}
\begin{empheq}[left={}\empheqlbrace]{align}
  & \frac{x_{n+1}-x_n}{h} = -f(x_n)+B\lambda_n,\\
  & y_n=Cx_n+D\lambda_n, t\geq 0,\\
  &  \lambda_{n+1}   \in -{\rm sign}(y_n), \;t\ge 0.
\end{empheq}
\end{subequations}
However, this explicit scheme usually does not converge as we can see in Figure 1 with $x_0=(1 \;\;2)^T$ and $h=0.04$. On the other hand, if we use Theorem \ref{mthe} with $x_0=(1 \;\;2)^T$  and $\gamma=0.5$, the Algorithm 1 converges very fast to $x^*=(0,0)$, which is an equilibrium of the given system (Figure 2).
 \begin{figure}[h!]\label{ex1}
\begin{center}
\includegraphics[scale=0.66]{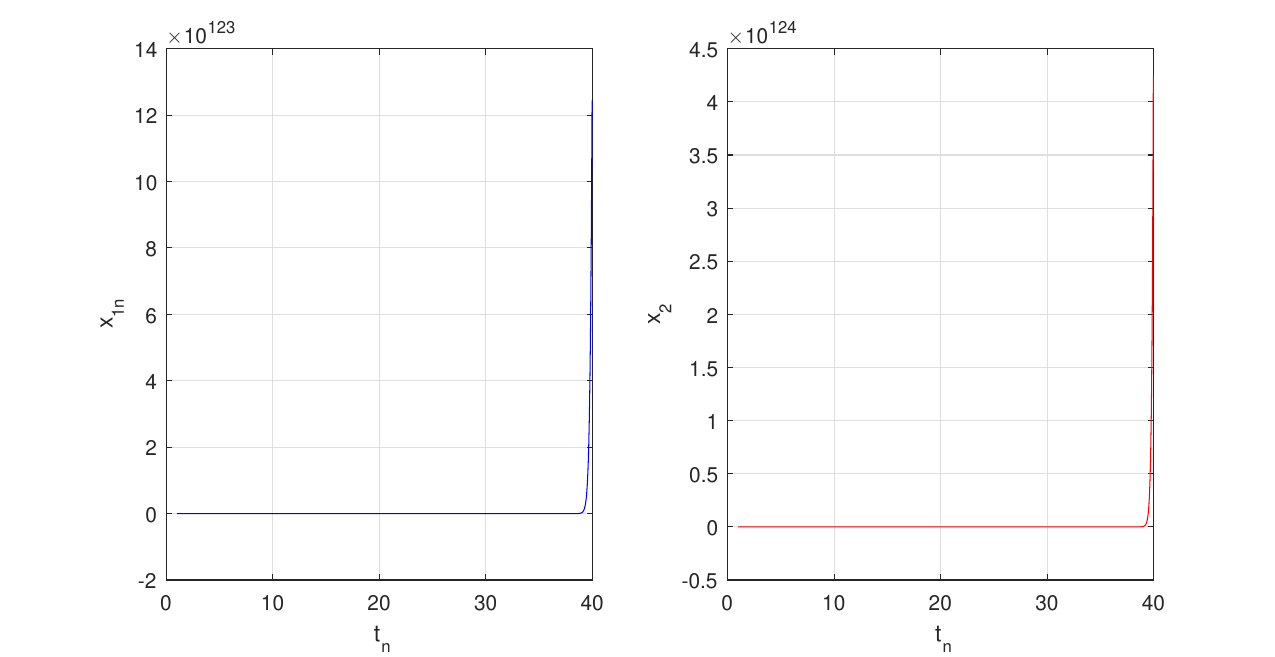}
\caption{The explicit scheme does not converge}
\end{center}
\end{figure}
 \begin{figure}[h!]\label{ex1}
\begin{center}
\includegraphics[scale=0.66]{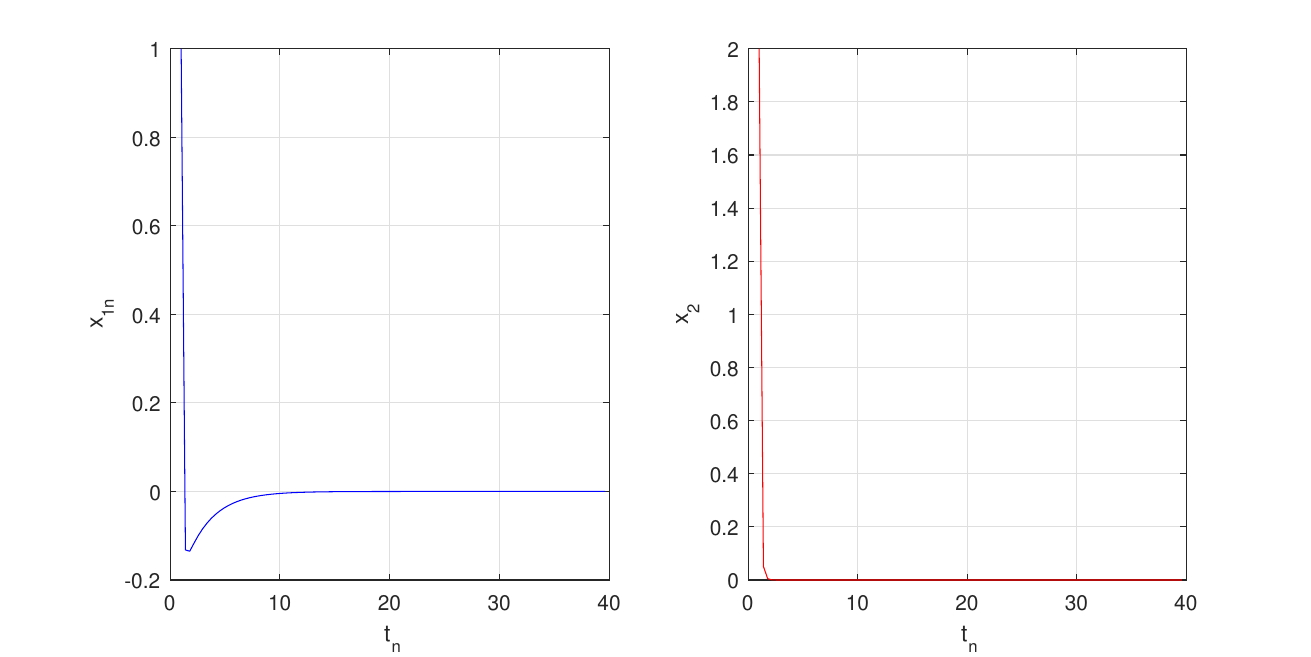}
\caption{The  convergence of Algorithm 1}
\end{center}
\end{figure}
\section{Conclusions}\label{sec7}
In this paper,  a method to compute   equilibria for  general set-valued Lur'e dynamical systems is provided. The obtained results can be applied to solve numerically equilibria for   set-valued Lur'e dynamical systems and   Nash {{quasi-equilibria}} in game theory.  It would be interesting if we could design algorithms for finding the  equilibria using less the implicit scheme than the semi-implicit scheme (\ref{discs}), which do not involve the resolvents of the composition operators to save the computational cost but still assuring the convergence.  This open question deserves further investigations. \\

\noindent $\mathbf{Acknowledgements}$\\

The authors would like to thank Professor Bernard Brogliato, a pioneer and an outstanding researcher,  for fruitful comments and suggestions.

\end{document}